\documentclass[11pt]{amsart}
\usepackage{amsmath}
\usepackage{latexsym}
\usepackage{amssymb,amsthm,amsfonts}
\usepackage{graphicx}
\usepackage{color} 

\newtheorem{defi}{Definition}[section]
\newtheorem{theorem}[defi]{Theorem}
\newtheorem{lemma}[defi]{Lemma}
\newtheorem{corollary}[defi]{Corollary}
\newtheorem{proposition}[defi]{Proposition}
\newtheorem{remark}[defi]{Remark}

\def\R{\mathbb{R}}

\def\N{\mathbb{N}}

\def\a{\alpha}

\def\e{\varepsilon}
\def\D{\Delta}

\def\la{\lambda}

\def\o{\omega}

\def\intr{\int_{\R^3}}
\def\dis{\displaystyle}
\begin{document}
\title[On a static Schr{\"o}dinger-Poisson-Slater
problem] {Ground and bound states for a static
Schr{\"o}dinger-Poisson-Slater problem}

\author{Isabella Ianni and David Ruiz}

\address{SISSA, via Beirut 2-4, 34014 Trieste (Italy) and Departamento de An\'alisis Matem\'atico, University of Granada, 18071 Granada (Spain)}

\thanks{D.R has been supported
by the Spanish Ministry of Science and Innovation under Grant
MTM2008-00988 and by J. Andaluc\'{\i}a (FQM 116).}

\email{ianni@sissa.it, daruiz@ugr.es}


\subjclass[2000]{}

\maketitle

\begin{abstract}
In this paper the following version of the
Schr{\"o}dinger-Poisson-Slater problem is studied:
$$ - \Delta u +  \left ( u^2 \star \frac{1}{|4\pi x|} \right
) u=\mu |u|^{p-1}u, $$ where $u: \R^3 \to \R$ and $\mu>0$.  The
case $p <2$ being already studied, we consider here $p \geq 2$.
For $p>2$ we study both the existence of ground and bound states.
It turns out that $p=2$ is critical in a certain sense, and will
be studied separately. Finally, we prove that radial solutions
satisfy a point-wise exponential decay at infinity for $p>2$.

\end{abstract}

\section{Introduction}

Recently, many papers have studied different versions of the
Schr{\"o}dinger-Poisson-$X^{\alpha}$ problem:

\begin{equation}\label{eq11} - \Delta u + \o u + \left ( u^2 \star \frac{1}{4\pi|x|} \right )
u= \mu |u|^{p-1}u, \ x \in \R^3, \end{equation} where $\mu>0$. The
interest on this problem stems from the Slater approximation of
the exchange term in the Hartree-Fock model, see \cite{slater}. In
this framework, $p=5/3$ and $\mu$ is the so-called Slater constant
(up to renormalization). However, other exponents have been used
in different approximations; for more information on the relevance
of these models and their deduction, we refer to \cite{bfortunato,
boka, bls, cornean, mauser}.

Our approach is variational, that is, we will look for solutions
of \eqref{eq11} as critical points of the corresponding energy
functional. From a mathematical point of view, this model presents
an interesting competition between local and nonlocal
nonlinearities. This interaction yields to some non expected
situations, as has been shown in the literature (see \cite{a-ruiz,
mugnai, mugnai2, daprile2, daprile3, kikuchi, kikutesis, pisani,
M3AS, JFA, preprint}). Other papers dealing with this kind of
variational problems are \cite{ianni, ianniVaira1, ianniVaira2,
mercuri, oscar, pisani, zhao, zhou}.

In this paper we consider the case $\o=0$. Recall that $\o$
corresponds to the phase of the standing wave for the
time-dependent equation; so, here we are searching static
solutions (not periodic ones). Following \cite{blions} we could
also say that this is a "zero mass" problem, since the linearized
operator at zero involves only the laplacian operator.

The static case has been motivated and studied in \cite{preprint}
as a limit profile for certain problems when $p <2$. Here we study
the existence of ground and bound states for \eqref{eq11} in the
case $p \geq 2$.

The absence of a phase term makes the usual Sobolev space
$H^1(\R^3)$ not to be a good framework for posing the problem
\eqref{eq11}. In \cite{preprint} the following space is
introduced:
$$ E=E(\R^3)= \{ u \in D^{1,2}(\R^3): \int_{\R^3} \int_{\R^3} \frac{u^2(x) u^2(y)}{|x-y|}\, dx \, dy < +\infty\}.$$
The double integral expression is the so-called Coulomb energy of
the wave, and has been very studied, see for instance \cite{lieb}.
In other words, $E(\R^3)$ is the space of functions in
$D^{1,2}(\R^3)$ such that the Coulomb energy of the charge is
finite. We also denote $E_r=E_r(\R^3)$ the subspace of radial
functions.

In \cite{preprint} it is shown that $E \subset L^q(\R^3)$ for all
$q \in [3,6]$, and the embedding is continuous. So, we have that
the energy functional $I_{\mu}: E \to \R$, \begin{equation}
\label{functional} I_{\mu}(u)= \frac 1 2 \int_{\R^3} |\nabla u|^2
\, dx + \frac{1}{4} \intr \intr \frac{u^2(x) u^2(y)}{|x-y|}\, dx
\, dy - \frac{\mu}{p+1} \intr |u|^{p+1}\,dx, \end{equation} is
well-defined and $C^1$ for $p\in[2,5]$. Moreover, its critical
points are solutions of
\begin{equation}\label{eq12} - \Delta u  + \left ( u^2 \star \frac{1}{4\pi|x|} \right )
u= \mu |u|^{p-1}u. \end{equation}

The above preliminary results are discussed in Section 2. Section
3 is devoted to the existence of ground states for $p>2$. The main
result is the following:

\begin{theorem} \label{teo1} Assume $p\in (2,5)$. Then there exists a ground
state for \eqref{eq12}, that is, there exists a positive solution
of \eqref{eq12} with minimal energy (among all nontrivial
solutions). \end{theorem}

We point out here that we do not know whether the ground state is
radially symmetric or not. Indeed, by restricting ourselves to
$E_r$, we can also show the existence of a radial ground state
(with minimal energy among all nontrivial radial solutions). But
we do not know if both solutions coincide.

The main problem in the proof of Theorem \ref{teo1} is the (PS)
property. First, for this problem it is not yet known if the
Palais-Smale sequences are bounded or not. To face this problem we
use a technique that dates back to Struwe and is usually named
"monotonicity trick" (see \cite{jeanjean, jtoland, struwe}). In so
doing, we can show the existence of bounded (PS) sequences for
almost all values $\mu >0$.

Secondly, bounded (PS) sequences could not converge, due to the
translation invariance of the problem. This problem is solved by
adapting the well-known arguments of concentration-compactness of
Lions (\cite{lions}). In this way, we obtain existence of ground
states for almost all values of $\mu$. With the help of a certain
"Pohozaev identity", we can extend the existence result to all
values of $\mu$. From the Pohozaev identity we also get
non-existence for $p\geq 5$ (see Corollary
\ref{corollnonexistence} in Section 3.)

In Section 4 we are concerned with the existence of multiple
(possibly sign-changing) solutions. Here we restrict ourselves to
the radial case, and work under a convenient constraint. By using
Krasnoselskii genus, we can prove the following result:
\begin{theorem}\label{teoBoundStates} \label{teo2} Assume $p\in (2,5)$. Then there
exist infinitely many radial bound states for \eqref{eq12}.
\end{theorem}

As we shall see in the final section, these solutions satisfy a
certain exponential decay, and in particular belong to
$L^2(\R^3)$. This justifies the name of "bound states" for these
solutions.

The case $p=2$ is critical because it presents a certain
invariance, and it is studied in Section 5. Indeed, given a
solution  $u$ of \eqref{eq12} and a parameter $\la$, the family of
functions $\la^2 u(\la x)$ is also a solution.

Then, restricting ourselves to the radial subspace $E_r,$ we can
prove the following result:

\begin{theorem} \label{teo3} There exists an increasing sequence
$\mu_k>0,$ $\mu_k \to +\infty$ such that the problem
\begin{equation}\label{eq13} - \Delta u  + \left ( u^2 \star \frac{1}{4\pi|x|} \right )
u= \mu_k |u|u \end{equation} has a radial solution $u_k$ (indeed,
there is a family of radial solutions given by the invariance of
the problem described above).
\end{theorem}

The above result can be thought of as a strongly nonlinear
eigenvalue problem. Indeed, the value $\mu_1$ is given by a
minimization process, in some aspects analogous to the first
eigenvalue. But this is achieved only in the radial space $E_r$.

In the last section we study the decay of the solutions that we
have found.  For $p>2$ and assuming radial symmetry we show that
the solutions of \eqref{eq12} satisfy an exponential decay
estimate at infinity. The result is obtained through comparison
arguments. As a consequence the solutions obtained in Theorem
\ref{teo2} belong to $L^2(\R^3)$, which is desirable from the
point of view of applications. We point out that this estimate
does not follow from arguments like in \cite{afm, bonheure,
swwillem}; here different arguments are to be used.

\section{Preliminaries}

We begin by enumerating some properties of the space $E$ and the
problem \eqref{eq12} that will be of use throughout the paper.
Next proposition has been proved in \cite{preprint}:

\begin{proposition} \label{app} Let us define, for any $u\in E$,
$$ \| u \|_E = \left ( \int_{\R^3}
|\nabla u(x)|^2\, dx  + \left ( \int_{\R^3} \int_{\R^3}
\frac{u^2(x) u^2(y)}{|x-y|}\, dx \, dy \right ) ^{1/2} \right
)^{1/2}.$$ Then, $\| \cdot \|_E$ is a norm, and $(E, \| \cdot
\|_E)$ is a uniformly convex Banach space. Moreover,
$C_0^{\infty}(\R^3)$ is dense in $E$, and also
$C_{0,r}^{\infty}(\R^3)$ is dense in $E_r$.
\end{proposition}

Let us define $\phi_u = \frac{1}{4 \pi |x|} \star u^2$; then, $u
\in E$ if and only if both $u$ and $\phi_u$ belong to
$D^{1,2}(\R^3)$. In such case, problem \eqref{eq12} can be
rewritten as a system in the following form:
\begin{equation} \label{eq21} \left\{ \begin{array}{l}
  -\D u +  \phi u= \mu u^p \\
  -\D \phi = u^2.
  \end{array}\right.\end{equation}

Moreover,
$$ \int_{\R^3} |\nabla \phi_u(x)|^2\, dx = \int_{\R^3} \phi_u(x) u(x)^2\, dx = \int_{\R^3}
\int_{\R^3} \frac{u^2(x) u^2(y)}{4 \pi |x-y|}\, dx \, dy.$$

By multiplying (a priori, formally, but it can be made rigorous by
truncating and using cut-off functions) the second equation in
\eqref{eq21} by $|u|$ and integrating, we obtain:
$$\intr |u|^3 = \intr (-\D \phi) |u| = \intr \langle \nabla \phi , \nabla
|u| \rangle . $$

We easily deduce the following inequality, that will be used may
times in what follows
\begin{equation} \label{lions} \intr |u|^3\leq \frac 1 2 \intr \left (|\nabla u|^2 +  |\nabla \phi|^2 \right).\end{equation}

By the above inequality and Sobolev inequality we conclude that $E
\subset L^q(\R^3)$ for any $q \in [3,6]$. Indeed, this range is
optimal and the embedding is continuous, see \cite{preprint}. As a
consequence, the functional $I_{\mu}: E \to \R$,
\begin{equation} \label{functional2} I_{\mu}(u)= \frac 1 2
\int_{\R^3} |\nabla u|^2 \, dx + \frac{1}{4} \intr \intr
\frac{u^2(x) u^2(y)}{|x-y|}\, dx \, dy - \frac{\mu}{p+1} \intr
|u|^{p+1}\,dx, \end{equation} is well-defined and $C^1$ for $p \in
[2,5]$.

In \cite{preprint} the following characterizations of the
convergences in $E$ is given:

\begin{lemma} \label{weak} Given a sequence $\{u_n\}$ in $E$, $u_n \to u$ in $E$ if and only if
$u_n \to u$ and $\phi_{u_n} \to \phi_u$ in $D^{1,2}(\R^3)$.

\noindent Moreover, $u_n \rightharpoonup u$ in $E$ if and only if
$u_n \rightharpoonup u$ in $D^{1,2}(\mathbb R^3)$ and $\int_{\R^3} \int_{\R^3}
\frac{u_n^2(x) u_n^2(y)}{ |x-y|}\, dx \, dy$ is bounded. In
such case, $\phi_{u_n} \rightharpoonup \phi_u$ in $D^{1,2}(\mathbb R^3)$.
\end{lemma}

For the sake of brevity, let us define:
$$T: E^4 \to \R, \ T(u,v,w,z)= \int_{\mathbb R^3}\int_{\mathbb
R^3}\frac{u(x) v(x)w(y)z(y)}{4 \pi |x-y|}dxdy.$$ Clearly, $T$ is a
continuous map, linear in each variable. Throughout the paper we
will need the following technical result:

\begin{lemma} \label{tecnico} Assume that we have three weakly convergent sequences in $E$, $u_n\rightharpoonup
u$, $v_n\rightharpoonup v$, $w_n\rightharpoonup w$, and $z \in E$.
Then:
$$T(u_n,v_n,w_n,z) \to T(u,v,w,z).$$

\end{lemma}

\begin{proof}

Observe that if two of the above sequences are constantly equal to
their respective limits, the conclusion holds immediately (we have
a continuous linear map applied to a weakly convergent sequence).

\medskip

{\bf Step 1} Suppose that $w_n=w$ for all $n \in \N$. Then:
$$ T(u_n,v_n, w,z)= T(u_n-u,v_n,w,z) + T(u,v_n,w,z).$$ By the above discussion, the
second right term converges to $T(u,v,w,z)$. Moreover, by using
Holder to the functions $(u_n(x)-u(x))w(y)$ and $v_n(x)z(y)$, we
have:
$$T(u_n-u,v_n,w,z)^2 \leq T(u_n-u,u_n-u,w,w) \, T(v_n,v_n,z,z).$$
The second term on the right being uniformly bounded, let us show
that the first term converges to $0$. Observe now that:
$$ T(u_n-u,u_n-u,w,w)= \intr \nabla \phi_{(u_n-u)} \cdot \nabla \phi_{w},$$
following the notation $\phi_u = \frac{1}{4 \pi |x|} \star u^2$.

Lemma \ref{weak} implies that $\phi_{(u_n-u)} \rightharpoonup 0$
in $D^{1,2}(\mathbb R^3)$, and this concludes the proof of Step 1.

\medskip

{\bf Step 2} Assume now that $u_n=u$ for all $n \in \N$. Then:
$$ T(u,v_n, w_n,z)= T(u,v_n-v,w_n,z) + T(u,v,w_n,z).$$
As above, the second right term converges to $T(u,v,w,z)$. We now
use Holder estimate to the functions $u(x) w_n(y)$ and
$(v_n(x)-v(x))z(y)$, to conclude:
$$ T(u,v,w_n,z)^2 \leq T(u,u, w_n, w_n)\,  T(v_n-v, v_n-v, z,z). $$
Observe now that the first right term is uniformly bounded and the
second converges to $0$ by the first step.

\medskip

{\bf Step 3 } Finally, we consider the general case.
$$ T(u_n,v_n, w_n,z)= T(u_n-u,v_n,w_n,z) + T(u,v_n,w_n,z).$$
By the second step, the second right term converges to
$T(u,v,w,z)$. With respect to the first term, we apply Holder
estimate to the functions $(u_n(x)-u(x))z(y)$ and $v_n(x) w_n(y)$:
$$ T(u_n-u,v_n,w_n,z)^2 \leq T(u_n-u,u_n-u, z, z) \, T(v_n, v_n, w_n,w_n). $$
The second right term is bounded and the first term converges to
zero thanks to Step 1.

\end{proof}

We also state here, for convenience of the reader, an adaptation
to the space $E$ of a result due to P.-L. Lions, see Lemma I.1 of
\cite{lions}:

\begin{lemma} \label{lemalions} Let $\{u_n\}$ a bounded sequence
in $E$, $q \in [3,6)$, and assume that
$$ \sup_{y \in \R^3} \int_{B(y,R)}|u_n|^q \to 0 \ \mbox{ for some } R>0. $$
Then $u_n \to 0$ in $L^{\alpha}(\R^3)$ for any $\alpha \in (3,6)$.
\end{lemma}

\begin{proof} By applying Lemma I.1 of \cite{lions} with $p=2$, we
obtain that $u_n \to 0$ in $L^{\alpha}(\mathbb R^3)$ for any $\alpha
\in (q,6)$. Recall now that $u_n$ is bounded in $E$, and hence in
$L^3(\mathbb R^3)$. We conclude by interpolation.
\end{proof}

To end up the section, we give a "Pohozaev-type" identity. This
identity is very close to the one given in \cite{mugnai2} for the
non-static case (that is, equation \eqref{eq11} with $\omega \neq
0$). The proof is exactly the same in this case and will be
skipped.

\begin{proposition} \label{poho} Let $p >0$ and $u \in E \cap H^2_{loc}(\mathbb R^3)$ be a weak solution of
\eqref{eq12}. Then:
\begin{equation} \label{eqpoho} \frac 1 2 \intr |\nabla u|^2 + \frac 5 4 \intr \intr
\frac{u^2(x) u^2(y)}{|x-y|}\, dx \, dy - \frac{3\mu}{p+1} \intr
|u|^{p+1}=0.\end{equation}

\end{proposition}

In particular, we have the following non-existence result, also
very close to that of \cite{mugnai2}:

\begin{corollary}\label{corollnonexistence} For $p \geq 5$, there is no solution $u \in E
\cap H^2_{loc}(\mathbb R^3)$ of problem \eqref{eq12}.
\end{corollary}

\section{Ground states in the case $p>2$}
Along this section we consider $p \in (2,5)$.  As we mentioned in
the introduction, we will look for solutions of \eqref{eq12} as
critical points of the functional $I_{\mu}$ defined in
\eqref{functional2}.

Let us define $M:E \to \R$ as: $$M[u]:=\int_{\mathbb R^3}|\nabla
u|^2dx+\int_{\mathbb R^3}\int_{\mathbb
R^3}\frac{u^2(x)u^2(y)}{|x-y|}dxdy.$$

Just by taking into account the definitions of $M$ and $\| \cdot
\|_E$, we can easily check that for any $u\in E$
\begin{equation} \label{M} \frac{1}{2}\|u\|^4_E \leq M[u]\leq
\|u\|^2_{E},\ \mbox{ if either } \|u\|_{E}\leq 1 \mbox{ or }
M[u]\leq 1 .\end{equation} The following estimate will be of use:
\begin{lemma}\label{LemmaNormaM}There exists $C>0$ such that
$$\|u\|^{p+1}_{L^{p+1}(\mathbb R^3)}\leq C M[u]^{\frac{2p-1}{3}},\ \mbox{ for all }u\in E.$$
\end{lemma}
\begin{proof}
Let $u_t(x):=t^2u(tx),$ for $t\in \mathbb R^+.$ By the continuity of
the embedding $E\hookrightarrow L^{p+1}(\mathbb R^3)$, we have:
\begin{equation}\label{stimaNorma}\int_{\mathbb R^3}|u|^{p+1}dx=t^{1-2p}\int_{\mathbb
R^3}|u_t|^{p+1}\leq C t^{1-2p} \|u_t\|_E^{p+1}.
\end{equation}
We fix now an appropriate $t$. For this scope observe that
$M[u_t]= t^3M[u],$ so choosing $t:=M[u]^{-\frac{1}{3}},$ it
follows that $M[u_t]=1$ and by \eqref{M} we obtain that
$\|u_t\|_E\leq \sqrt[4] 2.$ The conclusion follows substituting
this value of $t$ in (\ref{stimaNorma}).
\end{proof}

As a first consequence, we obtain a lower bound on $M[u]$ for
the solutions of \eqref{eq12}:

\begin{corollary} \label{hola}
There exists $\eta
> 0$ such that $M[u] > \eta $ for any nontrivial solution $u$ of
\eqref{eq12}.
\end{corollary}

\begin{proof}
By multiplying equation \eqref{eq12} by $u$ and integrating, we
obtain that $M[u] = \int_{\R^3} |u|^{p+1}$. Combining this with
the previous lemma, we have:
$$ \|u\|^{p+1}_{L^{p+1}(\mathbb R^3)}\leq C M[u]^{\frac{2p-1}{3}} \leq C \|u\|^{\frac{(2p-1)(p+1)}{3}}_{L^{p+1}(\mathbb R^3)}. $$
Since $p>2$, we conclude.
\end{proof}

We now turn our attention to the functional $I_{\mu}$, and show that
it satisfies the geometric properties of the mountain-pass theorem.

\begin{proposition} $I_{\mu}$ has a proper local minimum at $0$ and is unbounded from below.
\end{proposition}
\begin{proof}
We can estimate $I_{\mu}$ as:
\begin{equation}\label{stimaJ}I_{\mu}(u) \geq \frac{1}{4}M[u]-\frac{\mu}{p+1}\|u\|^{p+1}_{L^{p+1}(\mathbb
R^3)}.\end{equation} From (\ref{stimaJ}) and Lemma \ref{LemmaNormaM}
we get
\begin{equation}\label{stimaJFinale} I_{\mu}(u)\geq g(M[u])\end{equation}
where $g(s):=\frac{1}{4}s- \frac{C}{p+1}s^{\frac{2p-1}{3}} \geq
\frac 1 5 s$ for $s \in (0,\delta)$, being $\delta$ small enough.

Thanks to \eqref{M}, we can choose $\e\in (0,1)$ such that if
$\|u\|_E<\e$, $M[u]<\delta$, and then $I_{\mu}(u) \geq \frac 1 5
M[u] \geq \frac{1}{10} \|u\|_E^4$.

\medskip We now show that $I_{\mu}$ is unbounded below.
Fix $u\in E-\{0\}$ and define, for any $t>0$, $u_t(x)=t^2u(tx)$.
We compute:
$$I_{\mu}(u_t)= t^3\left [\frac{1}{2}\int_{\mathbb R^3}|\nabla u(x)|^2 dx +\frac{1}{4}\int_{\mathbb R^3}\int_{\mathbb R^3}\frac{u^2(x)u^2(y)}{|x-y|}dxdy \right]
-\mu \frac{t^{2p-1}}{p+1}  \int_{\mathbb R^3}|u(x)|^{p+1}dx.$$
Since $p>2$, $\lim_{t\rightarrow +\infty}J(u_t)=-\infty$.

\end{proof}

So, $I_{\mu}$ satisfies the geometric conditions of the
mountain-pass theorem of Ambrosetti-Rabinowitz (see \cite{a-rab}).
However, the main problem is that the (PS) condition does not
hold. If $p\geq 3$ it is easy to prove that (PS) sequences are
bounded in $E$, but this conclusion is not known for $p\in (2,3)$.

In order to face this difficulty, we use the so-called
"monotonicity trick", a method that dates back to Struwe
\cite{struwe} (see also \cite{jeanjean}). If fact, the name is
quite inconvenient since it has been proved not to depend on
monotonicity, see \cite{jtoland}.

Within this method, we need to use a min-max argument involving a
family of curves independent of $\mu$; this is at the core of the
technique. In so doing, one obtains solutions for almost all
$\mu$: after that we can complete the existence result by using
the Pohozaev identity. Similar reasonings have been used in
\cite{a-ruiz, j-tanaka, kikuchi}.

Let us fix $\e \in (0,1)$, and consider $\mu \in [\e, \e^{-1}]$.
Define the family of curves and the corresponding min-max value:
$$\Gamma=\left\{ \gamma \in C([0,1],E), \gamma(0)=0, \ I_{\e}(\gamma(1))<0
\right\},$$
$$c_{\mu}:=\inf_{\gamma\in\Gamma}\max_{t\in [0,1]}I_{\mu}(\gamma(t))>0, \ \mu \in [\e,\ \e^{-1}].$$

Clearly, if $\mu< \mu'$ we have that $c_{\mu} \geq c_{\mu'}$, and
hence we always have $c_{\mu} \geq c_{\e^{-1}}>0$. Observe also
that for any $\mu \in [\e,\e^{-1}]$ and any $\gamma \in \Gamma$,
$I_{\mu}(\gamma(1))<0$.

Our intention is to find a critical point at level $c_{\mu}$. By
next proposition, this solution will be a ground state.

\begin{proposition} \label{gs}
Let $\mu \in [\e,\e^{-1}]$ and $u \in E-\{0\}$ be a solution of
\eqref{eq12}. Then $I_{\mu}(u) \geq c_{\mu}$.
\end{proposition}

\begin{proof} Given such solution $u$, let us define again
$u_t(x)= t^2 u(t x)$, and $\gamma: \R \to E$, $\gamma(t)=u_t$.
Clearly $\gamma$ is a continuous curve in $E$ and $\gamma(0)=0$.
Moreover:
$$f(t)=I_{\mu}(\gamma(t))=t^3\left [\frac{1}{2}\int_{\mathbb R^3}|\nabla u(x)|^2 dx +\frac{1}{4}\int_{\mathbb R^3}\int_{\mathbb R^3}\frac{u^2(x)u^2(y)}{|x-y|}dxdy \right]
-\mu \frac{t^{2p-1}}{p+1}  \int_{\mathbb R^3}|u(x)|^{p+1}dx.$$ It
is easy to check that $f$ is $C^1$ and has a unique critical point
that corresponds to its maximum. Let us compute its derivative at
$t=1$:
$$f'(1)= \frac{3}{2}\int_{\mathbb R^3}|\nabla u(x)|^2 dx +\frac{3}{4}\int_{\mathbb R^3}\int_{\mathbb
R^3}\frac{u^2(x)u^2(y)}{|x-y|}dxdy -\mu \frac{2p-1}{p+1}
\int_{\mathbb R^3}|u(x)|^{p+1}dx.$$ Recall now that $u$ is a
solution, and hence verifies the Pohozaev-type identity
\eqref{eqpoho}. From this and from the identity $I_{\mu}'(u)(u)=0$
we deduce that $f'(1)=0$. That is:
$$ \max_{t \in \R} I_{\mu}(\gamma(t)) = I_{\mu}(u).$$

Since $\lim_{t \to +\infty} f(t)= -\infty$, we can take $M>0$ such
that $I_{\e}(\gamma(M))<0$. Reparametrizing $\gamma_0: [0,1] \to
E$, $\gamma_0(t)= \gamma(Mt)$, we obtain that $\gamma_0 \in
\Gamma$. Therefore, $c_{\mu} \leq I_{\mu}(u)$.

\end{proof}

We dedicate the rest of the section to prove that $c_{\mu}$ is a
critical value of $I_{\mu}$.

\begin{theorem} \label{monotonicity} There holds:
\begin{enumerate}
\item The map $[\e, \e^{-1}] \ni \mu \mapsto c_{\mu}$ is
nonincreasing and left continuous. In particular, it is almost
everywhere differentiable. Let us denote by $J \subset [\e,
\e^{-1}]$ the set of differentiability of $J$.

\item For any $\mu \in J$, there exists a bounded sequence
$\{u_n\} \subset E$ such that $I_{\mu}(u_n) \to c_{\mu}$,
$I_{\mu}'(u_n) \to 0$.

\end{enumerate}
\end{theorem}

The first assertion of the above theorem is quite evident. For the
proof of the second assertion see the general result of
\cite{jeanjean, jtoland} (see also Proposition 2.3 of
\cite{a-ruiz}).

Next proposition studies the behavior of bounded (PS) sequences:

\begin{proposition}\label{tipoWillemPag120}
Let $\{u_n\}\subset E$ be a bounded Palais-Smale of $I_{\mu}$
sequence at a certain level $c>0$. Then, up to a subsequence,
there exists $k \in \N \cup \{0\}$ and a finite sequence
$$\left(v_0, v_1,..,v_k\right)\subset E,\ v_i\not\equiv 0,\ \mbox{
for } i>0$$ of solutions of $$-\Delta u+\phi_u u=\mu u^{p}$$ and $k$
sequences $\{\xi_n^1\},..,\{\xi_n^k\}\subset \mathbb R^3,$ such that
$$
\begin{array}{lr}
\|u_n-v_0-\sum_{i=1}^k v_i(\cdot-\xi_n^i)\|_E\rightarrow 0\\
\\
|\xi_n^i|\rightarrow +\infty,\ |\xi_n^i-\xi_n^j|\rightarrow +\infty,\ i\neq j,\ \mbox{ as }n\rightarrow +\infty\\
\\
\sum_{i=0}^k I_{\mu} (v_i)=c ,\ M[u_n] \to \sum_{i=0}^k M[v_i].
\end{array}
$$
\end{proposition}
\begin{proof}

\textbf{ Step 1} Since $\{u_n\}$ is bounded and $E$ is a reflexive Banach space,
then, up to a subsequence, we may assume that $u_n \rightharpoonup
v_0$ in $E.$ Moreover $I_{\mu}'(v_0)=0$; indeed, if $\psi \in
C_0^{\infty}(\R^3)$,
$$
I_{\mu}'(u_n)(\psi)=\intr\nabla u_n\nabla \psi+\intr \phi_{u_n}u_n
\psi -\mu\intr |u_n|^{p-1}u_n \psi \to 0,$$
 $$\intr\nabla u_n\nabla \psi\rightarrow \intr\nabla v_0\nabla
\psi,\qquad \mbox{(since $u_n\rightharpoonup v_0$ in
$D^{1,2}(\mathbb R^3)$)},$$
$$\intr \phi_{u_n}u_n \psi \rightarrow \intr \phi_{v_0}v_0 \psi,\qquad \mbox{(from Lemma \ref{tecnico})},$$
$$\intr |u_n|^{p-1}u_n\psi \rightarrow \intr |v_0|^{p-1}v_0 \psi,$$
(observe that $u_n\rightarrow v_0$ in $L^p(K)$ when $K$ is
compact).

\medskip

Define  $$u_{n,1}:=u_n-v_0.$$ We claim now that
\begin{equation}\label{condizA} I_{\mu}(u_n) - I_{\mu}(u_{n,1}) \to I_{\mu}(v_0),\end{equation}
\begin{equation}\label{condizB} M[u_n] - M[u_{n,1}] \to M[v_0].\end{equation}

From weak convergence in $D^{1,2}(\mathbb R^3)$, it follows that
$\int_{\mathbb R^3}|\nabla u_n|^2\, dx - \int_{\mathbb R^3}|\nabla
(u_n-v_0)|^2\, dx \to \int_{\mathbb R^3}|\nabla v_0|^2\, dx$. By
passing to a convenient subsequence, if necessary, we can assume
that $u_n \to v_0$ almost everywhere. By the Brezis-Lieb lemma (see
for instance \cite{lieb, willem}, we have:
\begin{eqnarray*}
&&\int_{\mathbb R^3}|u_n|^{p+1}dx - \int_{\mathbb
R^3}|u_n-v_0|^{p+1} \, dx \to \int_{\mathbb R^3}|v_0|^{p+1}dx.
\end{eqnarray*}

We use the notation and the result of Lemma \ref{tecnico}, to
conclude:
$$ T(u_n-v_0, u_n-v_0, u_n-v_0, u_n-v_0)= T(u_n, u_n-v_0, u_n-v_0, u_n-v_0)+o(1) = $$ $$T(u_n, u_n, u_n-v_0, u_n-v_0)+o(1)=
T(u_n, u_n, u_n, u_n-v_0)+o(1)=$$$$T(u_n, u_n, u_n, u_n)-
T(v_0,v_0,v_0,v_0)+o(1).$$

This finishes the proof of the claim.

\medskip

If $u_{n,1}\rightarrow 0\mbox{ in }E$ we are done, since in this
case we have $I_{\mu}(u_{n,1})\to 0$, $M[u_{n,1}]\to 0$,
$$ I_{\mu}(v_0)=c,\ M[u_n] \to M[v_0].$$
Observe that in this case $v_0\not\equiv 0$  (since $c>0$).

Assume now that $u_{n,1}\not\rightarrow 0\mbox{ in }E$. Recall that
$\{u_n\}$ is a (PS) sequence and that $v_0$ is a solution, then
\begin{equation} \label{pufff} I_{\mu}'(u_n)(u_n)=M[u_n] -\mu \int_{\R^3} u_n^{p+1} \to 0
=M[v_0] - \mu \int_{\R^3} v_0^{p+1}. \end{equation}

Recall that $u_n \rightharpoonup v_0$ in $E$; by Lemma \ref{weak},
this implies that $u_n \rightharpoonup v_0$ and $\phi_{u_n}
\rightharpoonup \phi_{v_0}$ in $D^{1,2}(\mathbb R^3)$. Since $u_n
\not\rightarrow v_0$ in $E$, at least one of these convergences is
not strong in $D^{1,2}(\mathbb R^3)$, which implies that, up to a subsequence,
$$\lim_{n \to +\infty} M[u_n] = \lim_{n \to +\infty} \|u_n\|_{D^{1,2}(\mathbb R^3)}^2 +
\| \phi_{u_n}\|_{D^{1,2}\mathbb R^3}^2
> M[v_0].$$ Combining this with \eqref{pufff} we conclude that $u_n \not \to v_0$ in
$L^{p+1}(\mathbb R^3)$.

By Lemma \ref{lemalions}, given any $q\in [3,6)$, there exist
$\delta_1>0,$ $\{\xi_n^1\}\subset\mathbb R^3,$ such that
\begin{equation}\label{eqDaLemmaLionsInE}
\int_{B_1}|u_{n,1}(x+\xi_n^1)|^q \, dx\geq\delta_1>0.
\end{equation}

Since $u_{n,1} \rightharpoonup 0$, we have that $|\xi_n^1| \to +
\infty$.\\

\textbf{Step 2} Let us consider now the sequence
$\{u_{n,1}(\cdot+\xi_n^1)\}$. Obviously, it is a bounded (PS)
sequence at level $c-I_{\mu}(v_0)$ (recall \eqref{condizA}). Up to a
subsequence, we may assume that
$u_{n,1}(\cdot+\xi_n^1)\rightharpoonup v_1$ in $E$. As in Step 1 we
have that $v_1$ is a solution. By \eqref{eqDaLemmaLionsInE} we also
have that $v_1 \neq 0$.

Define $$u_{n,2}:=u_{n,1}-v_1(\cdot-\xi_n^1).$$

Arguing as in Step 1 and taking into account \eqref{condizA},
\eqref{condizB}, we obtain
$$ I_{\mu}(u_{n,2}) = I_{\mu}(u_{n,1})-I_{\mu}(v_1) + o(1)=I_{\mu}(u_n)-I_{\mu}(v_0)-I_{\mu}(v_1)+o(1).$$
Analogously, we also have
$$M[u_{n,2}] = M[u_{n,1}]-M[v_1] + o(1)=M[u_n]-M[v_0]-M[v_1]+o(1).$$

Observe that $u_{n,2} \rightharpoonup 0$ since both summands
converge weakly to zero, and $u_{n,2}(\cdot + \xi_n^1)
\rightharpoonup 0$ by the definition of $v_1$ (both weak
convergences must be understood in $E$).

If $u_{n,2}\rightarrow 0$ in $E,$ then we are done. Otherwise, as in
Step 1, we can show that $u_{n,2} \not \to 0$ in $L^{p+1}(\mathbb
R^3)$. By Lemma \ref{lemalions}, given any $q\in [3,6)$ there exist
$\delta_2>0,$ $\{\xi_n^2\}\subset\mathbb R^3$,  such that
\begin{equation}\label{eqDaLemmaLionsInE1}
\int_{B_1}|u_{n,2}(x+\xi_n^2)|^q \, dx\geq\delta_2>0.
\end{equation}

Since $u_{n,2} \rightharpoonup 0$ and $u_{n,2}(\cdot + \xi_n^1)
\rightharpoonup 0$ we deduce that $|\xi_n^2| \to +\infty$,
$|\xi_n^2-\xi_n^1|\to +\infty$. Therefore, up to a subsequence,
$u_{n,2}(\cdot + \xi_n^2) \rightharpoonup v_2 \neq 0$. We now
define:
$$ u_{n,3}=u_{n,2} - v_2(\cdot - \xi_n^2). $$

Iterating the above procedure we construct sequences
$\{u_{n,j}\}_{j=0,1,2,...}$  and $\{\xi_n^j\}_j,$ in the following
way
$$
\begin{array}{lr}
u_{n,\, j+1}=u_{n,j}-v_{j}\left(\cdot-\xi_n^j\right),\\
v_j:=\mbox{ weak}\lim u_{n,j}\left(\cdot+\xi_n^j\right),
\end{array}$$

\begin{eqnarray*}
&& I_{\mu}'(v_j)=0, \mbox{ for } j \geq 0 , \ \ v_j\not\equiv 0 \mbox{ for } j \geq 1,\\
&& I_{\mu}(u_{n,j}) = I_{\mu}(u_n)-\sum_{h=0}^{j-1}I_{\mu}(v_{h})+ o(1),\\
&&M[u_{n,j}] = M[u_n]-\sum_{h=0}^{j-1}M[v_{h}]+o(1).
\end{eqnarray*}
Now observe that $M[u_n]$ is bounded and $M[v_h] > \eta>0$ by
Corollary \ref{hola}. This implies that the iteration must stop at
a certain point, that is, for some $k$, $u_{n,k} \to 0$ in $E$.
This finishes the proof.

\end{proof}

\begin{corollary} \label{convergence} Let $u_n$ be a bounded (PS) sequence for
$I_{\mu}$ at level $c_{\mu}$. Then $u_n$ converges in $E$ (up to
translations) to a solution $u$, and $I_{\mu}(u)=c_{\mu}$.
\end{corollary}

\begin{proof}

We apply Proposition \ref{tipoWillemPag120}; in particular,
$$\sum_{i=0}^kI_{\mu} (v_i)=c_{\mu},$$ where $v_i$ are solutions of \eqref{eq12} and
only $v_0$ could be zero. By Proposition \ref{gs}, $I_{\mu}(v_i)
\geq c_{\mu}$ whenever $v_i \neq 0$. There are two possibilities
then: either $v_0 \neq 0$ and $k=0$, or $v_0=0$ and $k=1$. In the
first case, $v_0$ is a solution at level $c_{\mu}$ and $u_n \to v$
in $E$. In the latter, $v_1$ is a solution at level $c_{\mu}$ and
$u_n(\cdot + \xi_n^1) \to v_1$ in $E$.

\end{proof}

Next result concludes, together with Proposition \ref{gs}, the
proof of Theorem \ref{teo1}:

\begin{theorem} \label{jia} For any $\mu \in (\e, \e^{-1}]$, there exists a positive solution
$u \in E$ of \eqref{eq12}, and $I_{\mu}(u)=c_{\mu}$.\end{theorem}

\begin{proof}

First, assume that $\mu \in J$, where $J$ is defined in Theorem
\ref{monotonicity}. That theorem establishes the existence of a
bounded (PS) sequence at level $c_{\mu}$. By Corollary
\ref{convergence}, we conclude.

For general $\mu \in (\e, \e^{-1}]$, take a sequence $\{\mu_n\}
\subset J$, $\mu_n \to \mu$ increasingly, and $u_n$ critical points
of $I_{\mu_n}$ at level $c_{\mu_n}$. Since $\mu_n$ is increasing,
$c_{\mu_n} \to c_{\mu}$.

The functions $u_n$ satisfy both the equations
$I_{\mu_n}(u_n)=c_{\mu_n}$ and $I_{\mu_n}'(u_n)(u_n)=0$. Moreover,
they satisfy the Pohozaev identity \eqref{eqpoho}. We gather the
three equations in a system:
\begin{equation}
\left\{
\begin{array}{lr}
\frac{1}{2}A_n+\frac{1}{4}B_n-\frac{1}{p+1}C_n=c_{\mu_n}\\
\\
A_n+B_n-C_n=0\\
\\
\frac{1}{2}A_n+\frac{5}{4}B_n-\frac{3}{p+1}C_n=0,
\end{array}
\right.
\end{equation}
where $A_n=\intr |\nabla u_n|^2,$ $B_n=\intr \intr \frac{u_n^2(x)
u_n^2(y)}{|x-y|}\, dx \, dy$ and $C_n=\mu_n \intr |u_n|^{p+1}.$

Solving the above system, we get
$$
A_n=\frac{5p-7}{2(p-2)}\, c_{\mu_n},\ \ B_n=\frac{5-p}{p-2}\,
c_{\mu_n}, \ \ C_n=\frac{3(p+1)}{2(p-2)}\, c_{\mu_n}.$$

Since $c_{\mu_n}$ is bounded, we deduce that $A_n$, $B_n$ and $C_n$
must be bounded. In particular, the sequence $\{u_n\}$ is bounded in
$E$. Moreover,
$$ I_{\mu}(u_n)= I_{\mu_n}(u_n) + \frac{\mu_n-\mu}{p+1} \intr |u_n|^{p+1} = c_{\mu_n} + \frac{\mu_n-\mu}{p+1} \intr |u_n|^{p+1} \to c_{\mu}, $$
$$ I'_{\mu}(u_n)(v)= I'_{\mu_n}(u_n)(v)+ (\mu_n-\mu) \intr |u_n|^{p-1}u_n v \leq $$$$
|\mu_n-\mu| \|u_n\|_{L^{p+1}(\mathbb R^3)}^{p}
\|v\|_{L^{p+1}(\mathbb R^3)} \leq C |\mu_n-\mu| \|v\|_{E}.   $$ So,
$\{u_n\}$ is a bounded (PS) sequence for $I_{\mu}$ at level
$c_{\mu}$. By Corollary \ref{convergence} we conclude the existence
of a solution $u$.

We now prove that $u$ does not change sign. Recall that given any
$v \in E$, we denote $v_t(x)= t^2 v(t x)$. Define $f,\ g,\ h:
(0,+\infty) \to \R$ real functions as follows:
$$\begin{array}{l} f(t)= I_{\mu} (u_t) = \\ \\ \dis t^3\left [\frac{1}{2}\int_{\mathbb R^3}|\nabla u(x)|^2 dx +\frac{1}{4}\int_{\mathbb R^3}\int_{\mathbb R^3}\frac{u(x)^2u(y)^2}{|x-y|}dxdy \right]
-\mu \frac{t^{2p-1}}{p+1}  \int_{\mathbb
R^3}|u(x)|^{p+1}dx.\end{array}$$

$$\begin{array}{l} g(t)= I_{\mu} ((u^+)_t) = \\ \\ \dis t^3\left [\frac{1}{2}\int_{\mathbb R^3}|\nabla u^+(x)|^2 dx +\frac{1}{4}\int_{\mathbb R^3}
\int_{\mathbb R^3}\frac{u^+(x)^2u^+(y)^2}{|x-y|}dxdy \right] -\mu
\frac{t^{2p-1}}{p+1}  \int_{\mathbb
R^3}|u^+(x)|^{p+1}dx.\end{array}$$

$$\begin{array}{l} h(t)= I_{\mu} ((u^-)_t) = \\ \\ \dis t^3\left [\frac{1}{2}\int_{\mathbb R^3}|\nabla u^-(x)|^2 dx +\frac{1}{4}\int_{\mathbb R^3}
\int_{\mathbb R^3}\frac{u^-(x)^2 u^-(y)^2}{|x-y|}dxdy \right] -\mu
\frac{t^{2p-1}}{p+1}  \int_{\mathbb R^3}|u^-(x)|^{p+1}dx.
\end{array}$$

It is easy to check that $g(t)+h(t) \leq f(t)$ (the inequality
appears due to the nonlocal term). Reasoning as in Proposition
\ref{gs} we can show that $\max f = f(1) = I_{\mu}(u)=c_{\mu}$.
Take $t_1$, $t_2$ values at which the functions $g$, $h$ attain
their respective maxima. Assume, for instance, $t_1\leq t_2$; this
implies that $h(t_1) \geq 0$. So, $\max g = g(t_1)\leq g(t_1) +
h(t_1) \leq f(t_1) \leq \max f = c_{\mu}$. By the definition of
$c_{\mu}$, all previous inequalities must be equalities: in
particular, $h(t_1)=0$, and remember that $t_1 \leq t_2$. This is
only possible if $u^-=0$. If $t_1>t_2$, we can argue analogously
to prove that $u^+=0$.

So, up to a change of sign, we can assume $u\geq 0$. By the
maximum principle, we easily get that $u>0$.

\end{proof}

\begin{remark} We can also restrict ourselves
to the subspace of radial functions $E_r \subset E$ from the
beginning. In this way, one can prove that there exist radial
solutions at level $b_{\mu}$, defined as
$$b_{\mu}:=\inf_{\gamma\in\Lambda}\max_{t\in [0,1]}I_{\mu}(\gamma(t))>0, \ \mu \in [\e,\ \e^{-1}], $$
$$\Lambda=\left\{ \gamma \in C([0,1],E_r), \gamma(0)=0, \ I_{\e}(\gamma(1))<0
\right\}.$$

In this case the proof of the convergence of bounded (PS) sequences
is easier by compactness of the embedding $E_r \hookrightarrow
L^{p+1}(\mathbb R^3)$, see \cite{preprint}. Alternatively, we can
use Proposition \ref{tipoWillemPag120}; in a radial framework, $k$
must be equal to zero. The rest of the argument works as before.

These solutions have minimal energy among all solutions in $E_r$,
that is, an analogous to Proposition \ref{gs} holds for $E_r$ and
$b_{\mu}$.

We do not know if $b_{\mu}=c_{\mu}$ nor if both solutions may
coincide.

\end{remark}

\section{Bound states for $p>2$}
In this section we are concerned with the multiplicity of radial
(probably sign-changing) solutions of problem \eqref{eq12}. By
considering radial functions, we will be able to rule out the lack
of compactness due to the effect of translations.

Let us define:
\begin{equation}\label{Mtilde}\widetilde{\mathcal{M}}:=\left\{u\in E:
\int_{\mathbb R^3}|u|^{p+1}=1\right\}.\end{equation} We consider the
$C^1$ functional $J:E\rightarrow\mathbb R$
\begin{equation} \label{defJ} J(u)=\frac{1}{2}\int_{\mathbb R^3}|\nabla
u|^2dx+\frac{1}{4}\int_{\mathbb R^3}\phi_u u^2dx.
\end{equation}
Our scope is to find critical points of $J$ constrained on
$\widetilde{\mathcal M}$ in order to obtain solutions of
(\ref{eq12}).

We begin by the following lemma:
\begin{lemma}\label{lemmaJBoundedOnM} $J$ is bounded from below on $\widetilde{\mathcal{M}}$ and
$$\inf_{\widetilde{\mathcal{M}}} J>0.$$
\end{lemma}
\begin{proof}
From \cite{preprint} we now that:
$$ \| u \|_{L^{p+1}} \leq C \| u \|_{E}.$$

This implies that in $\widetilde{\mathcal{M}},$ the norm $\| \cdot
\|_E$ is bounded below. From the definition of $J$, the result
follows.
\end{proof}

From now on we restrict ourselves to the radial subspace $E_r$ and
prove Theorem \ref{teo2}. Obviously Lemma \ref{lemmaJBoundedOnM}
continues to hold just restricting $J$ to $E_r$ and substituting
the manifold $\widetilde{\mathcal {M}}$ with
\begin{equation} \label{defM}\mathcal{M}:=\left\{u\in E_r: \int_{{\mathbb
R}^{3}}|u|^{p+1}=1\right\}.\end{equation}

Since the manifold $\mathcal{M}$ is  symmetric ($u\in\mathcal
M\Rightarrow -u\in\mathcal M$) and $J$ is an even functional on
it, we are naturally led to apply techniques from
Ljusternik-Schnirelman category theory. Precisely (see definition
\eqref{MinMaxLevels} below) we use min-max characterizations of
critical values using the Krasnoselski genus (we refer for
instance to \cite{AMbook} for the definition of the genus and for
the theory related).

Let $\mathcal{A}$ to denote the set of closed and symmetric (with
respect to the origin) subsets of $E_r/\{0\}$ and let $\gamma(A)$
be the genus of a set $A\in\mathcal{A}.$ For any $k\geq 1$ let
\begin{equation}\label{MinMaxLevels} b_k := \inf_{A\in\Gamma_k}
\max_{x\in A} J(x),\end{equation} where
$$\Gamma_k:=\{A\subset \mathcal M: A\in\mathcal{A},\; A \mbox{
compact and } \gamma(A)\geq k\}.$$ Our scope is to show that each
$b_k$ is a critical value of $J|_{\mathcal M}.$ First let us
observe that $\left\{b_k\right\}_{k\geq 1}$  is well defined,
indeed next lemma implies that the set $\Gamma_k\neq \emptyset,$
for any $k\geq 1:$
\begin{lemma}\label{lemmaGenusM}The manifold $\mathcal{M}\in \mathcal A$ and $\gamma(\mathcal{M})=+\infty.$ \end{lemma}
\begin{proof}
The closure of $\mathcal{M}$ follows from the compactness of the
embedding $E_r\hookrightarrow L^{p+1}(\mathbb R^3).$

Observe that $\mathcal{M}$ is homeomorphic to the unit sphere $S$
of $E_r$ (through the homeomorphism $u\mapsto \lambda^2u(\lambda
x)$), and that $\gamma(S)=+\infty$ because $E_r$ is an infinite
dimensional Banach space. From the invariance of the genus by
homeomorphism it follows that $\gamma(\mathcal{M})=+\infty.$
\end{proof}
As we have already mentioned, the advantage of restricting
ourselves to $E_r$ is that we have more convenient compactness
properties. Indeed, we have:
\begin{lemma}\label{lemma(PS)OnM} $J$ satisfies the (PS) condition on $\mathcal{M}$.
\end{lemma}
\begin{proof}
Let $\{u_n\}\subset \mathcal{M}$ be a (PS) sequence on
$\mathcal{M}.$ $\{u_n\}$ is bounded because $J(u_n)$ is bounded by
assumption and it is easy to see that the functional $J$ is
coercive. Since $E_r$ is a reflexive Banach space, it follows that
up to a subsequence $u_n$ converges weakly in $E_r$ to a certain
$\bar{u}\in E_r.$ From the compactness of the embedding
$E_r\hookrightarrow L^{p+1}(\mathbb R^3)$ it follows that also
$\bar{u}\in \mathcal{M}.$ We claim that $u_n\rightarrow \bar{u}$
strongly in $E_r.$

Observe that for any $u \in \mathcal{M}$, $T_u \mathcal{M}=\{ v
\in E_r: \int_{\R^3} |u|^{p-1}uv=0\}$. So, we can define the
following projection onto $T_u \mathcal{M}$:
$$ P_u: E_r \to T_u \mathcal{M}, \ \ P_u(v)= v- u \int_{\R^3} |u|^{p-1} u v.$$

Take $w_n=P_{u_n}(u_n - \bar{u})$. Clearly, $w_n \in T_{u_n}
\mathcal{M}$ and is bounded in norm. Moreover, $w_n =
(u_n-\bar{u}) + \lambda_n u_n$, where $\lambda_n = - \int_{\R^3}
|u_n|^{p-1} u_n (u_n-\bar{u}) \to 0$.

Since $\{u_n\}$ is a (PS) sequence on $\mathcal{M},$ it follows
that
$$0 \leftarrow \left(J|_{\mathcal{M}}\right)'(u_n)(w_n) =
J'(u_n)(u_n-\bar{u}) + \lambda_n J'(u_n)(u_n).$$ We now use the
notation and the result of Lemma \ref{tecnico}:
$$ 0 \leftarrow J'(u_n)(u_n-\bar{u})= M[u_n]-\int_{\mathbb R^3}\nabla u_n\nabla \bar u-T(u_n,u_n,u_n,\bar u)=M[u_n]-M[\bar u]+o(1).$$
So we conclude that $M[u_n] \to M[\bar{u}]$, that is, $$ \| u_n
\|_{D^{1,2}(\mathbb R^3)}^2 + \| \phi_{u_n} \|_{D^{1,2}(\mathbb R^3)}^2 \to \| \bar{u} \|_{D^{1,2}(\mathbb R^3)}^2 + \|
\phi_{\bar{u}} \|_{D^{1,2}(\mathbb R^3)}^2.$$ By Lemma \ref{weak} $u_n\rightarrow u$ in
$E.$
\end{proof}

We can now complete the proof:
\begin{proof}[Proof of Theorem \ref{teo2}]
From Lemma \ref{lemma(PS)OnM} and Lemma \ref{lemmaJBoundedOnM} we
know that $J|_{\mathcal M}$ verifies the (PS) condition and is
bounded from below. From the genus theory (see e.g. \cite[Theorem
10.9]{AMbook}) it follows that each $b_k$ defined in
\eqref{MinMaxLevels} is a critical value for $J|_{\mathcal M}.$
Moreover the sequence $\left\{ b_k\right\}$  is obviously
non-decreasing and in particular $b_1=\inf_{\mathcal M} J,$ which
is strictly positive by Lemma \ref{lemmaJBoundedOnM}. Last, since
$\gamma(\mathcal M)=+\infty,$ we also know (see e.g. \cite[Theorem
10.10]{AMbook}) that $b_k\rightarrow \sup_{\mathcal M} J=+\infty.$

In conclusion, for any $k\geq 1,$ there exists (at least) a pair
$u_k,-u_k\in \mathcal M$ of radial solutions of the equation
\begin{equation} \label{eq-multip}
-\Delta u+\phi_u u=\mu_k |u|^{p-1}u \ \ \mbox{ in }\mathbb
R^3,\end{equation} with $J(u_k) = b_k \to +\infty$.

In particular $u_1>0,$  indeed $b_1=\min_{\mathcal M} J,$
$|u_1|\in \mathcal M$ and $J(|u_1|)=J(u_1)=b_1,$ hence we can
assume $u_1\geq 0.$ The strict inequality follows from the strong
maximum principle.

We now intend to get rid of the Lagrange multiplier $\mu_k$. First
of all, we have the following relation between $\mu_k$ and $b_k$.

\begin{lemma} \label{EstimateOfMu} $\mu_k = \frac{3(p+1)}{2p-1} b_k$. \end{lemma}

\begin{proof}[Proof of Lemma] Recall that $u_k \in \mathcal{M}$ is a
solution of \eqref{eq-multip}. Let us define $\alpha = \int_{\R^3}
|\nabla u_k|^2$, $\beta = \int_{\R^3} \int_{\R^3} \frac{u_k^2(x)
u_k^2(y)}{|x-y|} \, dx \, dy$. By multiplying equation
\eqref{eq-multip} by $u$ and integrating, we obtain: $\alpha +
\beta = \mu_k$. If we also take into account the equality
$J(u_k)=b_k$ and the Pohozaev identity \eqref{eqpoho}, we are led
with the system:
\begin{equation}
\left\{
\begin{array}{lr}
\alpha+ \beta =\mu_k, \\
\\
\frac{1}{2} \alpha + \frac{1}{4} \beta= b_k,\\
\\
\frac{1}{2} \alpha +\frac{5}{4} \beta =\mu_k.
\end{array}
\right.
\end{equation}

If we consider $\mu_k$ and $b_k$ as parameters, it is easy to
check that the above system is compatible only if $\mu_k =
\frac{3(p+1)}{2p-1}b_k$.

\end{proof}

In particular, the Lagrange multipliers $\mu_k$ are positive and
diverge as $k \to +\infty$. We now conclude the proof of Theorem
\ref{teo2}. Given $\lambda>0$, define again
$v_k(x)=\lambda^2u_k(\lambda x)$. Clearly, $v_k$ is a solution of:
$$ - \Delta v + \phi_v v = \lambda^{4-2p} \mu_k v^p.$$
By choosing $\lambda$ conveniently, we obtain a solution of
\eqref{eq12}.

\end{proof}

\section{The case $p=2$; proof of Theorem \ref{teo3}}
In this section we deal with the case $p=2.$ It differs from the
previous cases and turns out to be  critical  because, as already
observed in the introduction, it presents the following scaling
invariance: given a nontrivial solution $u$ of
\begin{equation}\label{eqP=2}-\Delta u+ \left ( u^2 \star
\frac{1}{4\pi|x|} \right ) u= \mu |u|u
\end{equation}
 and a parameter $\lambda\in\mathbb R,$ also the function
$\lambda^2u(\lambda x)$ is a solution.

Due to this invariance, for any solution $u$ of \eqref{eqP=2}, we
can re-scale it so that $\int_{\R^3} |u|^3=1$. Moreover, here we
will look for radial solutions only. As in the previous section,
our approach will be to find critical points of the functional $J$
on the manifold $\mathcal{M}$, where $J$ and $\mathcal{M}$ are
defined in \eqref{defJ}, \eqref{defM} respectively.

It is easy to check that all the procedure used in previous
section works also for $p=2$. Indeed, also the embedding $E_r
\hookrightarrow L^3(\R^3)$ is compact (see \cite{preprint}), which
is the essential tool to prove the (PS) property.

So, we obtain a sequence $b_k \to +\infty$ of critical values, and
a sequence of Lagrange multipliers $\mu_k$, with $\mu_k = 3 b_k$
(by Lemma \ref{EstimateOfMu}). Hence there exist solutions $\pm u_k$
of the problem:
\begin{equation} \label{eq-multip2}
-\Delta u+\phi_u u=\mu_k |u|u \ \ \mbox{ in }\mathbb
R^3,\end{equation}

The main difference is that now we cannot get rid of the Lagrange
multiplier as in the case $p>2$. In this way we conclude the
result of Theorem \ref{teo3}.

We point out that this is not a problem of the method of the
proof, but it is something intrinsic of the problem. As commented
previously, in the case $p=2$ the problem is invariant under the
transformation $t^2u(tx)$. It is quite reasonable then that
solutions appear only for certain values $\mu_k$, and in such case
we have a curve of solutions.

Indeed, it is easy to obtain the following non-existence result:

\begin{proposition}\label{teoP=2NonExistence}
For $\mu<2$ equation \eqref{eqP=2} has only the trivial solution
$u=0.$\end{proposition}

\begin{proof} We just multiply \eqref{eqP=2} by $u$, integrate, and recall inequality
\eqref{lions}:
$$ \mu \int_{\R^3} |u|^3 = \int_{\R^3} |\nabla u|^2 + |\nabla
\phi_u|^2 \geq 2 \intr |u|^3.$$

\end{proof}

\begin{remark} By multiplying \eqref{eqP=2} by u, integrating and using the
Pohozaev-type inequality \eqref{eqpoho} we get $I_{\mu}(u) = 0.$ So,
for $p = 2,$ all possible solutions of \eqref{eqP=2} have energy
equal to zero; this is another implication of the degeneracy of the
problem. Moreover, $I_{\mu}$ does not exhibit a mountain-pass
geometry; just observe that $0$ is not a proper local minimum for
$I_{\mu},$ since $I_{\mu}\equiv  0$ along any curve of solutions
$\lambda^2 u(\lambda x).$
\end{remark}

\begin{remark} One could consider
the minimization problem: $\inf \{J(u) : u \in\widetilde{\mathcal
M}\},$ see \eqref{Mtilde}, \eqref{defJ}. Observe that here we do
not assume radial symmetry. If $p > 2$ it can be proved that this
minimum exists; this was pointed out to us by Denis Bonheure in a
personal communication. Indeed, let us define the infimum:

$$ m_{\alpha} = \inf \left \{J(u):\ u\in E,
\ \int_{\mathbb R^3}|u|^{p+1}= \alpha \right\}.$$

By using the usual curve $ \lambda \mapsto \lambda^2u(\lambda x)$,
we can prove that $m_{\alpha+\beta} < m_{\alpha} + m_{\beta}$. So,
we can use the ideas of \cite[part 1, section 1]{lions} to avoid
dichotomy.

Instead, in Section 2 we have preferred to consider the free
functional $I_{\mu}.$ These arguments are more general and could
be useful to deal with nonlinearities different from $|u|^{p-1}u.$

The case $p=2$ seems to be much harder, and we do not know if there
exists a minimizer. If we define:
$$m_{\alpha} = \inf \{J(u) :\ u\in
E,\ \int_{\mathbb R^3} |u|^3 dx = \alpha\},$$ it is easy to check
that $m_{\alpha +\beta} = m_{\alpha} + m_{\beta}.$ So, the ideas of
\cite{lions} do not work, and indeed one can construct minimizing
sequences where dichotomy appears.
\end{remark}

\section{Decay estimates}
In this section we show that when $p>2,$ the radial solutions of
\eqref{eq12} have an  exponential decay  at infinity. Hence, in
particular, they belong to $L^2(\mathbb R^3).$ This result applies
both to the radial ground states and to the (non-positive) radial
bound states found in Section 3.

The main result of this section is the following:
\begin{theorem}\label{teodecay}
Assume $p\in (2,5)$ and let $u\in E$ be a radial solution of
\eqref{eq12}. Then there exist $C_1,C_2>0$ and $R>0$ such that
$$u(r)\leq C_1{r^{-\frac{3}{4}}}e^{-C_2\sqrt{r}} \ \mbox{ if }\; r>R.$$
\end{theorem}

\begin{proof}Let $u\in E$ be a radial  solution of the \eqref{eq12} and let as usual  $\phi_u(x)=\left ( u^2 \star \frac{1}{4\pi|x|} \right ).$
First of all, by a comparison argument,
\begin{equation}\label{lowerBoundForPhi}\phi_u(x) \geq \frac{c}{1+|x|}\mbox{ for some
}c>0.\end{equation}\\
We claim that for any $u \in E_r$ there exists a sequence $R_n \to
+\infty$ and a sequence $a_n \to 0$ such that $u(R_n) =
\frac{a_n}{R_n}$. The contrary would imply that there exists
$\epsilon>0$ and $R>0$ such that $|u(r)| \geq \frac{\epsilon}{r}$
for any $r>R$. But then $C \| u \|_E^3 \geq  \|u\|_{L^3(\mathbb R^3)}^3 \geq
\int_R^{+\infty} r^2 |u(r)|^3 \, dr = +\infty$, a
contradiction.\\

We now prove that there exists $R>0$ such that
\begin{equation}\label{inequalityCompareUPhi}
|u(x)|\leq\mu \phi_u(x) \  \mbox{ for }|x|>R.
\end{equation}

To prove that, we take $n$ large enough and use comparison
principles to compare $u$ and $\mu\phi_{u}$ in the complementary
of $B(0,R_n).$ We have
$$\left\{\begin{array}{lr}
-\Delta (u - \mu\phi_{u})= \mu|u|^{p-1}u- \phi_{u} u- \mu u^2 &
\mbox{ in
}|x|>R_n, \\
\\
(u-\mu\phi_u)(R_n) < 0 & \mbox{ in }|x|=R_n.
\end{array}\right.$$
\\
Multiplying the equation by $(u -\mu  \phi_{u})^+$  and integrating
in $\{|x|>R_n\}$ we get \begin{eqnarray*}&&\int_{|x|>R_n}|\nabla(u -
\mu\phi_u)^+|^2dx = \int_{|x|>R_n}\left\{\mu|u|^{p-1}u- \phi_u
u-\mu u^2\right\}(u -\mu \phi_u)^+dx\\
&&\qquad\leq\mu \int_{|x|>R_n}\left\{|u|^{p-1}u- u^2\right\}(u
-\mu\phi_u)^+dx\leq 0, \end{eqnarray*} where we used the fact that
$u>0$ when $(u -\mu \phi_u)^+\neq 0$ to eliminate the term $\phi_u
u$ and also
that $u \to 0$ at infinity  to obtain the last inequality.\\
Hence, $u \leq \mu\phi_u$ out of a certain fixed ball $B(0,R)$. In
the same way, we can show that $-u \leq \mu \phi_u$ and so
\eqref{inequalityCompareUPhi} is
proved.\\\\
From \eqref{lowerBoundForPhi}, \eqref{inequalityCompareUPhi} we
deduce that there exists $c'>0$ such that
\begin{equation}\label{ineqPhiU}
\phi_u(x)-\mu|u(x)|^{p-1}
\geq \frac{c'}{|x|} \ \mbox{ for }|x|>R.
\end{equation}

Inequality  \eqref{ineqPhiU} allows us to compare $u$ and $-u$ with
the radial solution $w$ of
$$\left\{
\begin{array}{lr}
-\Delta w+\frac{c'}{|x|}w=0\qquad\mbox{ if }|x|>R,\\
w= |u|\qquad\qquad\qquad\mbox{ if }|x|=R,\\
w\rightarrow 0\qquad\qquad\qquad\;\,\mbox{ if }|x|\rightarrow
+\infty.
\end{array}
\right.$$ More precisely, let us consider $u$ (the arguments for
$-u$ are similar), then we have
$$\left\{
\begin{array}{lr}
-\Delta (u-w) +\frac{c'}{|x|}(u-w)=\mu|u|^{p-1}u- \phi_u u + \frac{c'}{|x|}u\qquad \mbox{ if }|x|>R,\\
(u-w)\leq 0\qquad\qquad\qquad\qquad\qquad\qquad\qquad\qquad\quad\;\;\;\,\:\mbox{ if } |x|=R,\\
 u-w\rightarrow 0\qquad\qquad\qquad\qquad\qquad\qquad\qquad\quad\qquad\quad\;\;\,\mbox{ if }|x|\rightarrow
+\infty.
\end{array}
\right.
$$
Multiplying the equation by $(u -w)^+$  and integrating in
$\{|x|>R\}$ we get
\begin{eqnarray*}&&\int_{|x|>R}|\nabla(u -
w)^+|^2dx+\int_{|x|>R}\frac{c'}{|x|}\{(u-w)^+\}^2dx \\
&&\qquad\qquad\qquad =\int_{|x|>R}u\left(\mu|u|^{p-1}- \phi_u
+\frac{c'}{|x|}\right)(u -w)^+dx\leq 0,
\end{eqnarray*} where the last inequality follows from \eqref{ineqPhiU} and from the fact that,
since by weak maximum principle $w\geq 0,$ $u>0$ when $(u -w)^+\neq
0.$ Hence $u\leq w$ out of the ball $B(0,R).$ In the same way we
have $-u\leq w.$
\\\\
In conclusion we have proved that $|u| \leq w$ out of the ball
$B(0,R),$ but we know that $w$ has the exponential decay
$$w(r)\leq C \frac{1}{r^{\frac{3}{4}}}e^{-2c'\sqrt{r}} \ \mbox{ for }r> R'$$ (cfr. \cite[Section
4]{AmbrosettiMalchiodiRuiz}), for certain $C>0,$ $R'>0,$ hence the
theorem is proved.

\end{proof}

\begin{remark}
We conjecture that the same decay estimate holds also in the
nonradial case, as well as for $p=2$.
\end{remark}

\textbf{Acknowledgement:} I. Ianni would like to thank the members
of the Departamento de An\'alisis Matem\'atico of the University
of Granada, where this work was partially accomplished, for their
warm hospitality.

\end{document}